\documentclass[10pt]{amsart}
\usepackage{amsthm}          % theorem environments
\usepackage{tikz-cd}
\usepackage{amsmath}
\usepackage{amssymb}
\usepackage{physics}
\usepackage{enumerate}
\usepackage{comment}
\everymath{\displaystyle}
\usepackage{mathrsfs}

\newtheoremstyle{spaced}%
{6pt}     % space above
{12pt}    % space below
{\itshape} 
{}        % indent
{\bfseries} % theorem head font
{.}       % punctuation after head
{ }       % space after head
{}        % head spec

\theoremstyle{spaced}
\newtheorem{theorem}{Theorem}[section]
\newtheorem{lemma}[theorem]{Lemma}
\newtheorem{proposition}[theorem]{Proposition}
\newtheorem{corollary}[theorem]{Corollary}
\newtheorem{definition}[theorem]{Definition}

\newtheorem{maintheorem}{Theorem}

\numberwithin{equation}{section}

\newcommand{\RR}{\mathbb{R}}
\newcommand{\QQ}{\mathbb{Q}}

\newcommand{\BB}{\textbf{B}}
\newcommand{\clBB}{\overline{\textbf{B}}}

\newcommand{\cl}{\overline}
\newcommand{\Gr}{\operatorname{Gr}}

\newcommand{\CT}{\operatorname{\widehat{\textbf{T}}}}
\newcommand{\CN}{\operatorname{\widehat{\textbf{N}}}}

\newcommand{\NN}{\mathbb{N}}

\newcommand{\ind}{\mathbf{1}}

\newcommand{\BT}{\mathbf{T}}
\newcommand{\HpS}{\mathsf{Sel}_p(S)}
\newcommand{\Hp}{\mathsf{Sel}_p}
\newcommand{\Hq}{\mathsf{Sel}_q}
\newcommand{\HinftyS}{\mathsf{Sel}_\infty(S)}
\newcommand{\Hinfty}{\mathsf{Sel}_\infty}
\newcommand{\Hpone}{\mathsf{Sel}_{p,1}}
\newcommand{\Hpr}{\mathsf{Sel}_{p,r}}
\newcommand{\Hqs}{\mathsf{Sel}_{q,s}}
\newcommand{\Sel}{\mathsf{Sel}}

\title[]{The Clarke tangent and normal cone to decomposable sets in (products of) Lebesgue spaces}
\author{Petar Evgeniev}
\address{Faculty of mathematics and informatics, Sofia University}
\email{psevgeniev@uni-sofia.bg}

\thanks{ This research was supported by the Bulgarian National Science Fund under Grant No. KP-06-H92/6 (December 8, 2025); and partially supported by Sofia University Grant No. 80-10-31 ( 01.04.2026). %80-10-116 (28.05.2025) 
}

\begin{document}
	
	\begin{abstract}
		Let $(T,\Sigma,\mu)$ be a complete, $\sigma$-finite measure space, let $Y,Z$ be a separable Banach spaces, and let $S : T \rightrightarrows Z$ be a measurable multifunction with nonempty closed values. We study the relation between the Clarke tangent cone to the decomposable set $\Sel_p(S):=\{x \in L^p(T,Z) : x(t) \in S(t) \hspace{0.1cm} \text{a.e.}\}$ and the $L^p$-selections of the pointwise Clarke tangent cones. We prove that
		$$
		\CT_{\Sel_p(S)}(x) = \{v \in L^p(T,Z) : v(t) \in \CT_{S(t)}(x(t)) \hspace{0.1cm} \text{a.e.} \} \hspace{0.2cm} \text{for any $p \in [1, \infty)$},
		$$
		More generally, we consider this problem in $L^p(T,Y) \times L^r(T,Z), p,r \in [1,\infty),$ and denote $S : T \rightrightarrows Y \times Z$ and $\Sel_{p,r}(S):=\{(x,y) \in L^p(T,Y) \times L^r(T,Z) : (x(t),y(t)) \in S(t) \hspace{0.1cm} \text{a.e.}\},$
		 The global-to-pointwise inclusion is shown to hold for all exponents unconditionally, whereas the reverse inclusion is obtained from an equi-integrable correction condition. Consequently, if $p,r$ are finite it holds $$\CT_{\Sel_{p,r}(S)}(x,y)=\Sel_{p,r}\left\{t \mapsto \CT_{S(t)}(x(t),y(t))\right\} \hspace{0.2cm} \text{for any $(x,y) \in \Sel_{p,r}(S).$}$$ Whereas, if $(p,r)=(1,\infty)$ the equality holds if $\Sel_{\infty,1}(S)$ satisfies a variational correction condition (\textbf{V}). Condition shown to hold whenever $S(t):=\Gr F(t,\cdot),$ where $F : T \times Y \rightrightarrows Z$ is measurable in $t$ and $F(t,y) \subset F(t,x)+ \ell(t) \norm{y-x} \clBB_{Z}.$
		 
		 Applications are given to nonsmooth optimization with pointwise constraints, graphs of Nemytskii operators, and a multiplier rule for nonconvex integral programs is derived.
	\end{abstract}
	
	\maketitle
	
	\section{Formulation of the problem and literature review}
	Let $1 \leq p \leq \infty.$ By a pointwise-defined subset of a Lebesgue (Bochner) space $L^p(T,Z)$ we mean a set of the form 
	$$\HpS:=\{z \in L^p(T,Z) : z(t) \in S(t) \hspace{0.1cm} \text{a.e.}\}.$$
	Here $(T,\Sigma,\mu)$ is a complete, $\sigma$-finite measure space, $Z$ is a separable Banach space, and $S : T \rightrightarrows Z$ is a measurable multifunction with nonempty closed values. Pointwise-defined sets occur naturally in optimal control, the calculus of variations, and nonsmooth variational problems as pure or mixed state constraints; consequently, their variational geometry is important for the analysis of the corresponding constrained optimization problems.
	
	A basic question is whether the variational geometry of $\Sel_p(S)$ can be recovered from that of the pointwise sets $S(t).$ More precisely, let $x \in \Sel_p(S)$ and, for each $t \in T,$ let $C_{S(t)}(x(t))$ denote a tangent cone of interest to $S(t)$ at $x(t)$ in $Z.$ Let $C_{\Sel_p(S)}(x)$ denote the tangent cone, of the same type, to $\Sel_p(S)$ at $x$ in $L^p(T,Z).$ Is it true that
	
	\begin{equation} \label{eq:introduction-variational-objects}
		C_{\HpS}(x)= \{z \in L^p(T,Z) : z(t) \in C_{S(t)}(x(t)) \hspace{0.1cm} \text{a.e.}\}=:\Hp \left\{t \mapsto C_{S(t)}(x(t)) \right\} ?
	\end{equation}
	If equality fails, which of the two inclusions remains valid? The corresponding normal-cone question is formulated naturally in terms of selections in $L^q(T,Z^{*}),$ where $q$ is the conjugate exponent of $p$ and given that the dual of $L^p(T,Z)$ is identified with $L^q(T,Z^{*}).$
	
	For Bouligand tangent cones, a positive answer is known under the additional pointwise regularity assumption of derivability. More precisely, if the adjacent tangent cone to $S(t)$ at $x(t)$ coincides with the corresponding Bouligand tangent cone for a.e. $t \in T,$ then equality (\ref{eq:introduction-variational-objects}) holds for every $1 \leq p < \infty;$ see Theorem 8.5.1 and Corollary 8.5.2 in \cite{AubinFrankowska1990} and \cite{Aubin1987Update}. 
	
	For $p=\infty,$ an important result appears in the work of P\'ales and Zeidan \cite{PalesZeidan1999}. Under the additional assumptions $Z=\RR^n$ and $\mu(T)<\infty,$ they prove that the set of $L^\infty$-selections of the pointwise Clarke tangent cones is obtained as the relative $L^1$-closure of the Clarke tangent cone to the decomposable set $\Hinfty(S).$ This indicates that, in the case $p=\infty,$ the natural object is generally a closure of the global Clarke tangent cone rather than the cone itself; hence one should not expect (\ref{eq:introduction-variational-objects}) without additional assumptions.
	
		\begin{theorem}[Theorem 4, \cite{PalesZeidan1999} P\'ales , Zeidan]
		Let $(T,\Sigma,\mu)$ be a complete, finite measure space, let $S : T \rightrightarrows \RR^n$ be a measurable multifunction with nonempty closed values and $x \in \Hinfty(S).$ Then
		$$
		\cl{\CT_{\HinftyS}(x)}^{L^1} = \{ v \in L^\infty(T,\RR^n) : v(t) \in \CT_{S(t)}(x(t)), \hspace{0.1cm} \text{a.e.}\}.
		$$
		The same assertion holds for the adjacent tangent cone.
	\end{theorem}
	
	If $\mu(T)<\infty,$ then we can identify $L^\infty(T,\RR^n)$ as a subspace of $L^1(T,\RR^n).$ The considered $L^1$-closure is in the sense of the relative topology in $L^\infty$ inherited from $L^1.$
	
	Further, denote the embedding $J_1 : L^1 (T,\RR^n) \to [L^{\infty}(T,\RR^n)]^{*}$ given by 
	$$\left<J_1 \sigma ,v \right> =J_1 \sigma (v):=\int_T \left< \sigma(t),v(t) \right> \dd \mu(t) \hspace{0.2cm} \text{for each} \hspace{0.2cm} v \in L^\infty(T,\RR^n).$$
	
	Although $[L^\infty(T,\RR^n)]^{*}$ is larger than the range of $J_1,$ the countably additive normal functionals represented by elements of $L^1(T,\RR^n)$ admit the following complete pointwise characterization
	
	\begin{corollary}[Corollary 5, \cite{PalesZeidan1999} P\'ales, Zeidan]
		Let $(T,\Sigma,\mu)$ be a complete finite measure space, let $S : T \rightrightarrows \RR^n$ be a measurable multivalued mapping with nonempty closed values, and let $x \in \HinftyS.$
		If $\xi \in L^1(T,\RR^n),$ then $J_1 \xi \in \CN_{\HinftyS}(x)$ if and only if $\xi(t) \in \CN_{S(t)}(x(t))$ a.e. $t \in T.$
	\end{corollary}
	
	In their study of nonsmooth pure-state constraints, Krastanov and Ribarska obtained related partial results in the stationary setting, that is, when $S$ is a constant multifunction, for $p=\infty$ and $T=[a,b];$ see Propositions 3.1. and 3.2. in \cite{KrastanovRibarska2024PureStateEquality}.
	
	In recent works, Mehlitz and Wachsmuth \cite{MehlitzWachsmuth2018}, \cite{MehlitzWachsmuth2019WeakClosure} examined this problem for a hierarchy of tangent and normal cones in the case $1 < p < \infty,$ under the additional assumptions $Z=\RR^n,$ the sets $S(t)$ are derivable, and $(T,\Sigma,\mu)$ is nonatomic.
	Particular emphasis is put on limiting, Clarke, and Fr\'echet normal cones.
	
	For subsequent related work and applications see also \cite{Mehlitz2019SNC}, \cite{Wachsmuth2018Sobolev}, \cite{HarderWachsmuth2018SobolevNormal}, \cite{HarderWachsmuth2022MPCC}, \cite{PalesZeidan2000MeasurableConstraints}, \cite{PalesZeidan2004CriticalTangent} and \cite{CorellaWachsmuth2024}.
	
	Thus, the available results provide pointwise formulas under pointwise regularity assumptions, such as derivability or Clarke regularity; relative $L^1$-closure formulas in the $L^\infty$ setting; or formulas in finite-dimensional settings subject toadditional hypotheses. In the present paper we study the corresponding problem for the Clarke tangent cone without imposing regularity assumptions on the pointwise sets.
	To the author's knowledge, an exact pointwise representation of the Clarke tangent cone at this level of generality has not previously been available.
	
	We prove that, for every $1 \leq p < \infty$ and every separable Banach space $Z,$
	$$
	\CT_{\Sel_p(S)}(x) = \Sel_p \left\{ t \mapsto \CT_{S(t)}(x(t))\right\}.
	$$
	Whenever the dual of $L^p(T,Z)$ admits the usual identification, this yields the corresponding pointwise formula for the Clarke normal cone. We also consider this problem in products of Lebesgue spaces $L^p(T,Y) \times L^r(T,Z), 1\leq p,r \leq \infty,$ where the pointwise-defined sets become 
	$$\Sel_{p,r}(S):=\{(x,y) \in L^p(T,Y) \times L^r(T,Z) : (x(t),y(t)) \in S(t) \hspace{0.1cm} \text{a.e.}\}.$$ 
	The global-to-pointwise inclusion holds throughout the full range of exponents, whereas the converse direction is obtained by means of an equi-integrable correction condition, which holds automatically whenever the exponents are finite, and which provides exact formulas in the space $L^\infty \times L^1$ under natural variational correction assumption. The latter space is particularly important for the linearization of differential inclusions, the calculus of variations, and other classical problems involving trajectory-velocity pairings.
	
	For the pointwise-to-global direction, pointwise uniform tangency is transferred to the product Lebesgue space by an equi-integrable correction mechanism. For the converse inclusion, failure of pointwise Clarke tangency is quantified using Treiman's lemma. The set on which pointwise tangency fails is then decomposed into countably many measurable pieces, which are used to construct a nearby global base point contradicting global Clarke tangency.
	
	An immediate application is constrained optimization in $L^p(T,Z)$ with pointwise constraints of the form $x \in \Sel_p(S).$ By polarity, the tangent-cone representation yields a pointwise description of Clarke normal multipliers and hence provides a direct tool for deriving necessary optimality conditions for such problems.
	
	Although the present paper concerns tangent and normal cones to decomposable constraint sets rather than subdifferentials of integral functionals, it is related to the classical problem of representing generalized derivatives or subdifferentials of an integral functional in terms of measurable selections of the corresponding pointwise objects associated with the integrand. Such subdifferential formulas are generally more delicate than the tangent cone formula considered here, and a positive answer at the same level of generality cannot be expected.
	
	This motivates necessary conditions for optimization problems formulated directly in terms of Clarke tangent and normal cones; see, for example, the Lagrange multiplier rule of Bivas, Krastanov and Ribarska Theorem 3.3 in \cite{BivasKrastanovRibarska2020-tangential-transversality}.
	
	We recall this related line of work briefly. Let $f : T \times Z \to \RR \cup \{+\infty\}$ be a normal integrand, that is, measurable in $t$ and lower semicontinuous in $u \in Z.$ Its integral functional is the mapping
	$$u \mapsto I_f (u):=\inf\Big\{ \int_T a \dd \mu : a \in L^1(T), \hspace{0.1cm} a(t) \geq f(t,u(t)) \hspace{0.1cm} \text{a.e.}\Big\}.$$
	The convex case, in which $u \mapsto f(t,u)$ is convex for a.e. $t,$ was studied by Rockafellar \cite{Rockafellar1968IntegralsConvex}, \cite{Rockafellar1969MeasurableDependence}, \cite{Rockafellar1971IntegralsConvexII}, \cite{Rockafellar1971ConvexIntegralDuality} and \cite{Rockafellar1976Integral}.
	
	The Lipschitzian case was studied by Clarke \cite{Clarke1981}; and see also pp. 75--85 of \cite{ClarkeOld}.
	
	Giner \cite{Giner1998} studied the nonsmooth case for $Z=\RR^n$ and, under a growth condition on the normal integrand, proved the inclusion
	$$
	\partial^{\uparrow} I_f (x) \subset \{z \in L^q(T,\RR^n) : z(t) \in \partial^{\uparrow} f(t,x(t)) \hspace{0.1cm} \text{a.e.}\} =:\Sel_q \left\{ t \mapsto \partial^{\uparrow} f(t,x(t))\right\},
	$$
	where $1 \leq p < \infty, \hspace{0.1cm} \frac{1}{p} + \frac{1}{q} =1,$ and $\partial^{\uparrow}$ is Rockafellar's generalized directional subdifferential \cite{Rockafellar1979}, \cite{Rockafellar1980}.
	This direction was further developed in several works of Giner, concerning
	Lipschitzian properties \cite{Giner2007}, calmness and contingent subgradients
	\cite{Giner2009}, and the relation between Clarke and limiting subdifferentials \cite{Giner2017ClarkeLimiting}.
	
	More recently Giner and Penot provided a systematic treatment of the subject. They study integral functionals on $L^p(T,Z),$ where $(T,\Sigma,\mu)$ is a complete $\sigma$-finite measure space, $1 < p < \infty$ and $Z$ is a separable Banach space.  Under regularity assumptions on the integrand, they prove exact equalities for various subdifferentials; see Theorems 39-41 in
	\cite{GinerPenot2018}. In particular as a consequence of Theorem 39 in \cite{GinerPenot2018}; if the values of $S : T \rightrightarrows Z$ along $x \in \HpS$ are regular in the sense that the Clarke tangent cone to $S(t)$ at $x(t)$ coincides with the Bouligand tangent cone to $S(t)$ at $x(t)$ for a.e. $t \in T,$ then (\ref{eq:introduction-variational-objects}) holds for the Clarke and Bouligand tangent cones. Indeed, take the integrand
	$$f(t,u)=\iota_{S(t)}(u):=\begin{cases}
		0, \hspace{0.2cm} u \in S(t) \\
		+\infty, \hspace{0.2cm} u \notin S(t)
	\end{cases}.
	$$
	The conditions $(C_p)$ and $(G),$ in Theorem 39 in \cite{GinerPenot2018}, are satisfied for this $f.$ The remaining assumption is C-D regularity. For this $f$ it is equivalent to the following equality
	$$\CT_{S(t)}(x(t)) = \BT_{S(t)}(x(t)) \hspace{0.2cm} \text{a.e.},$$
	where $\CT$ and $\BT$ denote the Clarke and Bouligand tangent cones, respectively.
	The upper-integral functional for this $f$ is $I_{f}(x) = \iota_{\HpS}(x).$ Thus the conclusion is that
	$
	\CT_{\HpS}(x) = \Sel_p \left\{  t \mapsto \CT_{S(t)}(x(t)) \right\}  = \BT_{\Sel_p (S)}(x).
	$
	
		\begin{theorem}[Giner, Penot; Theorem 39 in \cite{GinerPenot2018}]
		Let $(T,\Sigma,\mu)$ be a complete, $\sigma$-finite measure space, $Z$ be a separable Banach space, $1 < p < \infty$ and $S : T \rightrightarrows Z$ be measurable multifunction with nonempty closed values. If $x \in \Sel_p(S)$ and $\CT_{S(t)}(x(t)) = \BT_{S(t)}(x(t))$ a.e., then
		$$
		\CT_{\HpS}(x) = \{ v \in L^p(T,Z) : v(t) \in \CT_{S(t)}(x(t)) \hspace{0.1cm} \text{a.e.}\} = \BT_{\HpS}(x).
		$$
	\end{theorem}
	
	\section{Main contributions of the paper}
	%---------
	
	We show that, for the Clarke tangent cone, the pointwise regularity assumption can be removed. More precisely the main contribution of the current paper is the following:
	
	\begin{maintheorem} \label{thorem:Main-Lp-theorem}
		Let $(T,\Sigma,\mu)$ be a complete $\sigma$-finite measure space, let $Z$ be a separable Banach space, $1 \leq p < \infty$ and $S : T \rightrightarrows Z$ be a measurable multivalued mapping with nonempty closed values. If $x \in \HpS,$ then
		$$
		\CT_{\HpS}(x)=\{v \in L^p(T,Z) : v(t) \in \CT_{S(t)}(x(t)) \hspace{0.1cm} \text{a.e.}\}.
		$$
	\end{maintheorem}
	
	The space $L^\infty(T,Y) \times L^1(T,Z)$ is of particular interest. It arises naturally when one studies tangent objects associated with trajectory-velocity pairs, in particular in the linearization of differential inclusions,  see e.g. \cite{BivasKrastanovRibarska2023}.
	
	A set $K \subset L^1(T,Z)$ is equi-integrable if for any decreasing sequence of measurable sets with $A_n \searrow \emptyset$ a.e. it holds 
	$$\sup_{k \in K} \int_{A_n} \norm{k(t)} \dd \mu(t)  \longrightarrow 0 \hspace{0.2cm} \text{as $n \to \infty.$}$$
	Every ralatively weakly compact set in $L^1(T,Z)$ is equi-integrable, see e.g. \cite[p.~104, Theorem~4]{DiestelUhl}
	
	We introduce the following variational condition (\textbf{V}):
	For every $M>0,$ there exists $\delta, \lambda > 0$ and an equi-integrable set $K \subset L^1(T,Z)$ such that, for every $\eta \in [0,\lambda]$ it holds
	$$
	\Sel_{\infty,1}(S) \cap \left( (x,y) + \delta_M \clBB_{L^\infty \times L^1} \right) + \eta \left( \mathcal{U}_M \times \{0\} \right) \subset \Sel_{\infty,1}(S) - \eta ( \{0\} \times K),
	$$
	where $\mathcal{U}_M:=\left\{ u \in M.\clBB_{L^\infty} : \exists v \in L^1(T,Z), \hspace{0.1cm} (u(t),v(t)) \in \CT_{S(t)}(x(t),y(t)) \hspace{0.1cm} \text{a.e.}  \right\}.$
	
	 We refer the reader to \cite{KrastanovRibarska2021PAFA}, where correction conditions, involving several variables, are also studied and applied.
	
	\begin{proposition}
		Let $(T,\Sigma,\mu)$ be a complete, $\sigma$-finite measure space, let $Y,Z$ be separable Banach spaces and $S: T \rightrightarrows Y \times Z$ be a measurable with nonempty closed values. If $\Sel_{\infty,1}(S)$ satisfies the variational condition (\textbf{V}) at $(x,y) \in \Sel_{\infty,1}(S),$ then
		$$
		\CT_{\Sel_{\infty,1}}(x,y) = \{ (u,v) \in L^\infty(T,Y) \times L^1(T,Z) : (u(t),v(t)) \in \CT_{S(t)}(x(t),y(t)) \hspace{0.1cm} \text{a.e.}\}.
		$$
	\end{proposition}

	(\textbf{V}) is implied by the following simpler property (\textbf{Vstrong}) of $\Sel_{\infty,1}(S)$ at $(x,y).$ There exists $\varepsilon, \delta,\lambda>0$  and an equi-integrable $K \subset L^1(T,Z)$ such that for every $\eta \in [0,\lambda]$
	$$
	\Sel_{\infty,1}(S) \cap \left( (x,y) + \delta \clBB_{L^\infty \times L^1}\right) + \eta (\varepsilon \clBB_{L^\infty} \times \{0\})) \subset \Sel_{\infty,1}(S) - \eta ( \{0\} \times K).
	$$
	
	(\textbf{Vstrong}) is satisfied when $S(t)=\Gr F(t,\cdot)$ where $F : T \times Y \rightrightarrows Z$ is measurable in $t$ and $F(t,y) \subset F(t,z) + \ell(t) \norm{y-z}\clBB_{Z} $ for some $\ell \in L^1(T).$
	
	We note that in the latter case, when $Y,Z$ finite dimensional and $T$ compact interval, the equality of the Clarke tangent cones was proven by Mikhail Minkov in his Master's thesis \cite{MikailMinkov}.
	
	We prove the two inclusions separately. We begin with a construction of uniform tangent sets to decomposable sets in $L^p(T,Z).$ 
	
	Since many applications involve more than one Lebesgue-space variable, we study the interaction of two such variables and formulate our results for the full range of exponents, including infinity. By taking one of the variables to be identically zero, we recover the corresponding one-variable results.

	\section{Preliminaries and notation}
	Throughout, $Z$ denotes a separable Banach space; $\clBB_{Z}$ and $\BB_{Z}$ denote its closed and open unit balls, respectively. The symbol $d_C(x)$ for the distance from $x$ to $C.$ $\widehat{\partial}f$ denotes the Clarke subdifferential \cite{ClarkeOld} of a locally Lipschitz real-valued function $f.$
	
	\begin{definition}
		Let $X$ be a Banach space,let $C \subset X$ be closed, and let $x \in C.$ We say that $v$ is Clarke tangent to $C$ at $x$ if, for every $\varepsilon>0,$ there exist $\delta>0$ and $\lambda>0$ such that, for every $y \in C \cap (x + \delta \clBB_{X}),$
		$$
		C \cap [ y + \eta (v + \varepsilon \clBB_{X})] \neq \emptyset
		$$
		for all $\eta \in [0,\lambda].$
	\end{definition}
	We denote the set of all Clarke tangent vectors by $\CT_{C}(x).$ The corresponding normal cone is defined by polarity, $\CN_{C}(x):=\left( \CT_{C}(x) \right)^{0}.$
	
	\begin{definition}
		Let $X$ be a Banach space, let $C \subset X$ be closed, and let $x \in C.$ We say that a subset $D \subset X$ is a uniform tangent set to $C$ at $x,$ if for every $\varepsilon>0,$ there exist $\delta>0$ and $\lambda>0$ such that, for every $v \in D$ and every $y \in C \cap (x + \delta \clBB_{X}),$
		$$
		C \cap [y + \eta( v + \varepsilon \clBB_{X})] \neq \emptyset
		$$
		for all $\eta \in[0,\lambda].$
	\end{definition}
	
	\begin{lemma}[Treiman, Lemma 2.1 \cite{TreimanLemma}]
		Let $X$ be a Banach space, $C$ its closed subset and $x_0$ belonging to $C.$ Then $y \notin \CT_{C}(x_0)$ if and only if there exists $\varepsilon_0>0$ such that for every $\delta>0$ there exists $\lambda>0, x' \in C \cap ( x_0 + \delta \clBB_{X})$ with
		$$
		C \cap [ x' + (0,\lambda)(y + \varepsilon_{0} \clBB_{X}) ] = \emptyset.
		$$
	\end{lemma}
	
	\begin{definition}
		Let $X$ be a Banach space, let $C \subset X$ be closed, and let $x \in C.$ We say that $v$ is a Bouligand tangent to $C$ at $x,$ if there exists a sequence of positive reals $\tau_k \searrow 0$ and a sequence $v_k \to v$ such that $x + \tau_k v_k \in C$ for every $k \in \NN.$ 
	\end{definition}
	The set of all Bouligand tangent vectors to $C$ at $x$ is denoted by $\BT_{C}(x).$
	
	\begin{definition}[\cite{CastaingValadier1977}]
		Let $(T,\Sigma,\mu)$ be a complete $\sigma$-finite measure space, let $Z$ be a separable Banach space and $S : T \rightrightarrows Z$ be a multifunction with nonempty closed values. We say that $S$ is measurable if one, and hence all,of the following equivalent conditions hold:
		\begin{enumerate}[(a)]
			\item for every open $U \subset Z$ it's big  preimage $\{t \in T : U \cap S(t) \neq \emptyset\}$ is $\Sigma$-measurable;
			
			\item there exists a sequence $\{s_n\}_{n \in \NN}$ of $\Sigma$-measurable functions $s_n : T \to Z$ such that $s_n(t) \in S(t)$ and $S(t)= \cl{\{s_n(t)\}_{n \in \NN}}$ for a.e. $t \in T;$ such a sequence is called a Castaing representation of $S;$
			
			\item the graph $\Gr S =\{(t,x) \in T \times Z  : x \in S(t)\}$ is $\Sigma \otimes \mathcal{B}(Z)$-measurable, where $\mathcal{B}(Z)$ denotes the Borel $\sigma$-algebra of $Z;$
		\end{enumerate}
	\end{definition}
	
	A function $z :T \to Z$ is measurable if, for every open set $U \subset Z,$ its preimage $z^{-1}[U]:=\{t \in T : z(t) \in U\}$ is $\Sigma$-measurable. It is said that $z$ is Bochner measurable if there exists a sequence of simple functions whose pointwise limit coincides with $z$ for almost all values of $t.$ Whenever $Z$ is separable, these two notions are equivalent. We denote by $L^p(T,Z)$ the (Lebesgue--Bochner) space, consisting of all classes of measurable functions $z:T \to Z$ such that 
	$$\norm{z}_{L^p}:=\left( \int_T \norm{z(t)}^p \dd \mu(t) \right)^{1/p} < \infty.$$
	For a detailed exposition and study of Bochner spaces see \cite{CembranosMendoza1997}. 
	
	\begin{lemma}
		Let $(T,\Sigma,\mu)$ be a measure space, let $Z$ be a separable Banach space, and let $S : T \rightrightarrows Z$ be a measurable multifunction with nonempty closed values. If $x : T \to Z$ is measurable,  then the multifunctions
		\begin{align*}
			t \mapsto \CT_{S(t)}(x(t)) \hspace{0.5cm} \text{and} \hspace{0.5cm} t \mapsto \BT_{S(t)}(x(t))
		\end{align*}
	
	\end{lemma}
	
	\begin{lemma}
		Let $(T,\Sigma,\mu)$ be a $\sigma$-finite measure space, $Z$ be a separable Banach space, and $1 \leq p \leq \infty.$
		The following are equivalent for $S : T \rightrightarrows Z$ measurable multifunction with nonempty closed values:
		\begin{enumerate}[(a)]
			\item $\Hp(S) \neq \emptyset;$
			
			\item the function $t \mapsto \operatorname{dist}(0, S(t))$ belongs to $L^p(T,[0,\infty));$
			
			\item There exists $r(\cdot) \in L^p(T,[0,\infty))$ such that $S(t) \cap r(t) \clBB_{Z} \neq \emptyset$ a.e.
		\end{enumerate}
	\end{lemma}
	
	\begin{definition}
	We say that a set $M \subset L^p(T,Z)$ is decomposable if for every $x,y \in M$ and every $A \in \Sigma$ it holds $\ind_{A} x + \ind_{T \setminus A} y \in M.$ Here $\ind_{A}$ stands for the indicator function of a measurable set $A.$
	\end{definition}
	
	Decomposable sets were studied by Rockafellar \cite{Rockafellar1968IntegralsConvex}, Hiai and Umegaki \cite{HiaiUmegaki1977} and others.
	
	The following characterization of decomposable sets allows one to study them through their $\Hp$-representation.
	
	\begin{lemma}[Theorem 3.1. in \cite{HiaiUmegaki1977} and Theorem 6.4.6 in \cite{KyritsiYiallourouPapageorgiou2009}]
		Let $(T,\Sigma, \mu)$ be a $\sigma$-finite measure space, let $Z$ be a separable Banach space, and $1 \leq p < \infty.$ A closed nonempty set $M \subset L^p(T,Z)$ is decomposable if and only if $M=\HpS$ for some measurable multifunction $S  : T \rightrightarrows Z$ with nonempty closed values.
	\end{lemma}

	\section{The main result in $L^p \times L^r$ for $1 \leq p,r \leq \infty$}
	We endow $L^p(T,Y) \times L^r(T,Z)$ with $\norm{\cdot}_{p,r}:=\norm{\cdot}_{L^p} + \norm{\cdot}_{L^r}$ and $Y \times Z$ with $\norm{\cdot}_{Y} + \norm{\cdot}_{Z}.$ For the purpose of unifying some arguments for finite and infinite exponents we denote $$\vartheta_s(a):= \begin{cases} a^{1/s}, \hspace{0.2cm} s < \infty \\
		1, \hspace{0.6cm} s=\infty \end{cases}.$$
	This function is nondecreasing.
		
		The following is an $L^p$-notion of equi-integrability in the sense of Marle \cite{Marle1974}. See also Giner \cite{Giner2012}.
		
	\begin{definition}
		$R \subset L^p(T,Y) \times L^r(T,Z)$ is equi-integrable if for any sequence of measurable sets $A_n \searrow \emptyset$ a.e. it holds $\sup_{z \in R} \norm{\ind_{A_n} z}_{L^p \times L^r} \to 0$ as $n \to \infty.$
	\end{definition}
	
	\begin{lemma}
		If $T$ is a complete $\sigma$-finite measure space, then $R \subset L^p(T,Y) \times L^r(T,Z)$ is equi-integrable if and only if the following two properties hold
		\begin{enumerate}[(i)]
			\item $\forall \varepsilon>0, \hspace{0.1cm} \exists K \in \Sigma, \hspace{0.1cm} \mu(K) < \infty, \hspace{0.1cm} \sup_{z \in R} \norm{\ind_{T \setminus K} z}_{L^p \times L^r} < \varepsilon;$
			
			\item $\forall \varepsilon>0 \hspace{0.1cm} \exists \alpha>0$ s.t. $\mu(A) < \alpha \Rightarrow \sup_{z \in R} \norm{\ind_{A} z}_{L^p \times L^r} < \varepsilon.$
		\end{enumerate}
	\end{lemma}
	In $L^\infty(T,Y)$ the only possible equi-integrable set is $\{0\}.$ Thus when two $L^p$-variables are involved its usefull for the variable with finite exponent to correct (possibly) the infinite one.
	
	\begin{definition}
		Let $X$ be a Banach space, $H,V,R$ its subsets and $x_0 \in H.$ The set $H$ is $R$-correctable along $V$ at $x_0$ if there exists $\delta,\lambda>0$ such that
		$$
		H \cap (x_0 + \delta \clBB_{X}) + \eta V \subset H - \eta R, \hspace{0.5cm} \forall \eta \in [0,\lambda].
		$$
	\end{definition}
	This variational condition extends compact epi-Lipschitzness \cite{BorweinStrojwas1985}. A closed set $H$ is compactly epi-Lipschitz at $x_0$ exactly if $H$ is $K$-correctable along $V:=\varepsilon_0 \clBB_{X}$ for some $\varepsilon_0>0$ and some compact $K.$
	
	The following result encapsulates the argument for the incusion of pointwise selections-to-global in the main theorem.

	\begin{maintheorem} \label{theorem:osnoven-rezultat-potochovo-kum-globalno}
		Let $(T,\Sigma,\mu)$ be a complete $\sigma$-finite measure space, let $1 \leq p,r \leq \infty$ with atleast one finite, let $Y,Z$ be separable Banach spaces, let $S : T \rightrightarrows Y \times Z$ be measurable with nonempty closed values and $(x,y) \in \Sel_{p,r}(S).$
		
		If $D : T \rightrightarrows Y \times Z$ is measurable with nonempty closed values such that
		\begin{enumerate}[(a)]
			\item $D(t)$ is uniform tangent to $S(t)$ at $(x(t),y(t))$ a.e.
			
			\item $\Sel_{p,r}(S)$ is $R$-correctable along $\Sel_{p,r}(D)$ at $(x,y),$ for some equi-integrable family $R \subset L^p(T,Y) \times L^r(T,Z).$
		\end{enumerate}
		Then $\Sel_{p,r}(D)$ is uniform tangent to $\Sel_{p,r}(S)$ at $(x,y).$ 
	\end{maintheorem}
	\begin{proof}
		Fix $\varepsilon>0.$ By the equi-integrability of $R$ choose $K \in \Sigma$ with $0< \mu(K) < \infty$ and $\alpha>0$ such that
		\begin{align} \label{eq:prod-equi-1}
			\sup_{q \in R} \norm{\ind_{T \setminus K} q}_{L^p \times L^r} &< \frac{\varepsilon}{8}, \\
			\mu(A) < \alpha \Longrightarrow \sup_{q \in R} \norm{\ind_{A} q}_{L^p \times L^r} &< \frac{\varepsilon}{8}.
		\end{align}
		Choose $\beta>0$ such that
		\begin{equation} \label{eq:prod-equi-2}
			2 \beta. \left( \vartheta_p (\mu(K)) + \vartheta_r(\mu(K))\right) < \frac{\varepsilon}{2}.
		\end{equation}
		Let $s_j : T \to Y \times Z, d_i : T \to Y \times Z,$ for $i,j \in \NN,$ be Castaing representations of $S$ and $D,$ respectively.
		
		For $m \in \NN,$ let $E_m$ consist of all $t \in T$ such that for every $i,j \in \NN$ and every $\theta \in \left[0, \frac{1}{m}\right] \cap \QQ$ the following implication holds
		\begin{equation} \label{eq:product-equi-4}
			\norm{s_{j}(t) - (x(t),y(t))}_{Y \times Z} < \frac{1}{m} \Longrightarrow d_{S(t)}( s_j(t) + \theta d_i(t) ) \leq \beta \theta.
		\end{equation}
		Then $E_m$ are measurable, as they are obtained as a countably many intersections and unions of measurable sets, and increasing $E_m \subset E_{m+1}.$
		
		Let $T_0$ be the set of all $t \in T$ such that $D(t)$ is uniform tangent set to $S(t)$ at $(x(t),y(t)).$ By assumption $\mu(T \setminus T_0)=0.$ Notice that 
		$$T_0 \subset \bigcup_{m=1}^{\infty} E_m.$$ 
		Indeed, let $t \in T_0.$ For the fixed $\beta>0$ there exists $\delta_t, \lambda_t>0$ such that
		\begin{equation} \label{eq:product-equi-5}
			d_{S(t)}(z+ \eta w) \leq \beta \eta,
		\end{equation}
		whenever
		$$
		z \in S(t), \hspace{0.2cm} \norm{z - (x(t),y(t))}_{Y \times Z} < \delta_t, \hspace{0.2cm} w \in D(t), \hspace{0.2cm} \eta \in [0,\lambda_t].
		$$
		Now we choose $m$ so large that 
		$$\frac{2}{m}<\delta_{t} \hspace{0.2cm} \text{and} \hspace{0.2cm} \frac{1}{m} < \lambda_t.$$ Then (\ref{eq:product-equi-5}) gives the conclusion in (\ref{eq:product-equi-4}) and proves $t \in E_m.$ Since $K$ has finite measure and $E_m \nearrow T,$ up to a null set, we can choose $m_0$ such that
		\begin{equation} \label{eq:product-equi-6}
			\mu(K \setminus E_{m_0}) < \frac{\alpha}{2}.
		\end{equation}
		Let $\rho,\lambda>0$ be supplied by the $R$-correctability of $\Sel_{p,r}(S)$ along $\Sel_{p,r}(D)$ at $(x,y).$ Put $\lambda_0 :=\min\left\{ \lambda, \frac{1}{m_0}\right\}$ and choose $\delta>0$ so small that $\delta < \min \left\{ \rho, \frac{1}{2m_0}\right\}$ and
		\begin{equation} \label{eq:product-equi-7}
			\sum_{s \in \{p,r\} \setminus \{\infty\}} (2m_{0})^{s} \delta^s < \frac{\alpha}{2}.
		\end{equation}
		
		Fix arbitrary
		$$
		(u,v) \in \Sel_{p,r}(D), \hspace{0.2cm} (a,b) \in \Sel_{p,r}(S) \cap ( (x,y) + \delta \clBB_{L^p \times L^r} ), \hspace{0.2cm} \eta \in (0,\lambda_0].
		$$
		Define $B:=B_a \cup B_b,$ where
		\begin{align*}
		B_{a}:&=\Big\{t \in K : \norm{a(t)- x(t)}_{Y} \geq \frac{1}{2m_0}\Big\}, \\
		 B_{b}:&=\Big\{t \in K : \norm{b(t)-y(t)}_{Z} \geq \frac{1}{2m_0}\Big\}.
		\end{align*}
		If $p=\infty,$ then $B_{a}=\emptyset;$ otherwise $p<\infty$ and by Chebyshev's inequality
		$$
		\mu(B_a) \leq (2m_{0})^{p} \norm{x-a}^p_{L^p}.
		$$
		Likewise, $B_b=\emptyset$ if $r=\infty,$ whereas
		$$
		\mu(B_b) \leq (2m_{0})^r \norm{y-b}_{L^r}^{r}.
		$$
		for $r<\infty.$
		Hence by $(\ref{eq:product-equi-7})$ we conclude $\mu(B)  < \frac{\alpha}{2}.$
		Set
		$$
		G:=(K \cap E_{m_0}) \setminus B, \hspace{0.2cm} A:=K \setminus G.
		$$
		Since
		$A=K \setminus G = K \setminus \left(  ( K \cap E_{m_0}) \setminus B\right) = (K \setminus E_{m_0}) \cup B$ we have 
		$$\mu(A) \leq \mu(K \setminus E_{m_0}) + \mu(B) < \alpha.$$
		We can then estimate that for every $t \in G$
		\begin{equation} \label{eq:product-equi-10}
			d_{S(t)}( (a(t),b(t)) + \eta (u(t),v(t)) ) \leq \beta \eta.
		\end{equation}
		Indeed, if $t \in G,$ then 
		$$\norm{(a(t),b(t)) - (x(t),y(t))}_{Y \times Z} = \norm{a(t) - x(t)}_{Y} + \norm{b(t) - y(t)}_{Z}< \frac{1}{m_0}.$$
		Choose sequences
		$$
		s_{j_k}(t) \to (a(t),b(t)), \hspace{0.2cm} d_{i_k}(t) \to (u(t),v(t)), \hspace{0.2cm} \theta_k \in \left[0, \frac{1}{m_0} \right] \cap \QQ, \hspace{0.2cm} \theta_k \to \eta.
		$$
		For all sufficiently large $k$ 
		$$\norm{s_{j_k}(t) - (x(t),y(t))}_{Y \times Z}< \frac{1}{m_0}.$$ 
		Since $t \in E_{m_0}$ it holds
		$$
		d_{S(t)}(s_{j_k}(t) + \theta_k d_{i_k}(t)) \leq \beta \theta_k.
		$$
		Passing to the limit we obtain (\ref{eq:product-equi-10}) on $G.$
		
		For the fixed $\eta,$ define $\Phi_\eta : G \rightrightarrows Y \times Z$ by
		$$
		\Phi_\eta(t):=S(t) \cap \left( (a(t),b(t)) + \eta (u(t),v(t)) + 2 \beta \eta \clBB_{Y \times Z} \right).
		$$
		The estimate (\ref{eq:product-equi-10}) implies that $\Phi_{\eta}$ has nonempty values. Therefore $\Phi_\eta$ is measurable with nonempty closed values. Let $(a_{\eta}^{G},b_{\eta}^{G}) : G \to Y \times Z$ be a measurable selector of $\Phi_\eta.$ Then $(a_{\eta}^G(t), b_{\eta}^{G}(t)) \in S(t)$ and
		\begin{equation} \label{eq:product-equi-13}
		\norm{a_{\eta}^{G}(t) - a(t) - \eta u(t)}_{Y} + \norm{b_{\eta}^{G}(t) - b(t) - \eta v(t)}_{Z} \leq 2 \beta \eta,
		\end{equation}
		for a.e. $t \in G.$
		
		By the $R$-correctabillity, there exists $r_{\eta} \in R$ and $(\widetilde{a}_{\eta},\widetilde{b}_{\eta}) \in \Sel_{p,r}(S)$ such that
		\begin{equation} \label{eq:product-equi-14}
			(\widetilde{a}_{\eta},\widetilde{b}_{\eta}) = (a,b) + \eta ( (u,v) + r_{\eta}).
		\end{equation}
		Define the patched selectors
		\begin{align*}
			\widehat{a}_{\eta}&:=\ind_{G} a_{\eta}^{G} + \ind_{T \setminus G} \widetilde{a}_{\eta}, \\
			\widehat{b}_{\eta} &:= \ind_{G} b_{\eta}^{G} + \ind_{T \setminus G} \widetilde{b}_{\eta}.
		\end{align*}
		Then $(\widehat{a}_{\eta},\widehat{b}_{\eta}) \in \Sel_{p,r}(S).$
		
		On $G,$ the estimate (\ref{eq:product-equi-13}) and $G \subset K$ give
		\begin{align*}
			\norm{\ind_{G}(\widehat{a}_{\eta} - a - \eta u)}_{L^p} + \norm{\ind_{G}(\widehat{b}_{\eta}-b - \eta v)}_{L^r} 
			&\leq 2 \beta \eta ( \vartheta_p ( \mu(K)) + \vartheta_r(\mu(K))) \\
			&< \frac{\varepsilon \eta}{2}.
		\end{align*}
		On $T \setminus G$ the patched selector coincides with the corrected one and hence
		\begin{equation} \label{eq:product-equi-15}
			\ind_{T \setminus G}( (\widehat{a}_{\eta},\widehat{b}_{\eta}) -(a,b)- \eta (u,v)) = \eta \ind_{T \setminus G} r_{\eta}.
		\end{equation}
		Since $T \setminus G = (T \setminus K) \cup A$ and $\mu(A) <\alpha$ we estimate
		\begin{align*}
			\norm{\ind_{T \setminus G}( (\widehat{a},\widehat{b}) -(a,b) - \eta (u,v))}_{p,r} &= \eta \norm{\ind_{T \setminus G} . r_{\eta}}_{p,r} \\
			&\leq \eta. \left(\norm{\ind_{T \setminus K}. r_{\eta}}_{p,r} + \norm{\ind_{A} r_{\eta}}_{p,r} \right) \\
			&< \eta \left(\frac{\varepsilon}{8} + \frac{\varepsilon}{8} \right) \\
			&<\frac{\eta \varepsilon}{2}.
		\end{align*}
		Where $\norm{\ind_{T \setminus G} w_{\eta}}_{p,r} \leq \norm{w_{\eta}}_{p,r} \leq \frac{\varepsilon}{4}.$ The assertion for $\eta=0$ is immediate.
		Combining the estimates on $G$ and $T \setminus G$ we derive
		$$
		\norm{(\widehat{a}_{\eta},\widehat{b}_{\eta}) - (a,b) - \eta (u,v)}_{L^p \times L^r} < \varepsilon \eta.
		$$
		This and the fact that $\delta$ and $\lambda_0$ are independent of $(u,v), (a,b)$ and $\eta$ prove $\Sel_{p,r}(D)$ is a uniform tangent set to $\Sel_{p,r}(S)$ at $(x,y).$
	\end{proof}
	The following is a typical possition for the case of finite exponents.
	\begin{proposition}
		Let $(T,\Sigma,\mu)$ be a complete, $\sigma$-finite measure space, let $Y,Z$ be separable Banach spaces, let $S : T \rightrightarrows Y \times Z$ be measurable with nonempty closed values. 
		
		If $1 \leq p,r < \infty, \hspace{0.1cm} (x,y) \in \Sel_{p,r}(S)$ and $D : T \rightrightarrows Y \times Z$ is measurable with nonempty closed values such that
		\begin{enumerate}[(i)]
			\item $D(t)$ is uniform tangent set to $S(t)$ at $(x(t),y(t))$ a.e.;
			
			\item There exists $g \in L^p(T,[0,\infty))$ and $h \in L^r (T, [0,\infty))$ such that 
			$$D(t) \subset (g(t),h(t)) \clBB_{Y \times Z};$$
		\end{enumerate}
		Then 
		$\Sel_{p,r}(D)$ is uniform tangent set to $\Sel_{p,r}(S)$ at $(x,y);$
	\end{proposition}
	
	\begin{proof}
		Under the stated assumptions, for finite exponents $\Sel_{p,r}(S)$ is $R$-correctable along $\Sel_{p,r}(D)$ at $(x,y)$ with $R:=-\Sel_{p,r}(D).$ In particular $\Sel_{p,r}(S)$ is first-order $R$-correctable along $\Sel_{p,r}(D)$ at $(x,y).$  The domination condition on $D(t)$ ensures that $\Sel_{p,r}(D)$ is equi-integrable. Indeed, for $d \in \Sel_{p,r}(D)$ pick $q:=-d.$ Then $z+\eta (d+q)=z \in \Sel_{p,r}(S).$ Furhter for $A \in \Sigma$ it holds 
		$$\sup_{q \in -\Sel_{p,r}(D)} \norm{\ind_{A} q}_{L^p \times L^r} \leq \norm{\ind_{A} g}_{L^p} + \norm{\ind_{A} h}_{L^r}.$$
		Combined with the fact that for $p,r < \infty$ and $A_n \searrow \emptyset$ it holds $\norm{\ind_{A_n} g}_{L^p} \to 0.$ Similarly for the $L^r$-variable.  
		And analogous result holds on $T \setminus K.$ Thus  Theorem \ref{theorem:osnoven-rezultat-potochovo-kum-globalno} yields the conclusion.
	\end{proof}
	From this, we are able to obtain the the pointwise-to-global inclusion whenever the exponents are finite. A variat to which has a direct proof in \cite{PHDGiner}.
	\begin{corollary}
		Let $(T,\Sigma,\mu)$ be a complete, $\sigma$-finite measure space, let $Y,Z$ be separable Banach spaces, let $S : T \rightrightarrows Y \times Z$ be measurable with nonempty closed values. 
		If $1 \leq p,r < \infty, \hspace{0.1cm} (x,y) \in \Sel_{p,r}(S),$ then
		$$
		\Sel_{p,r}\left\{ t \mapsto \CT_{S(t)(x(t),y(t))}\right\} \subset \CT_{\Sel_{p,r}(S)}(x,y).
		$$
	\end{corollary}
	\begin{proof}
		Let $(u,v) \in 	\Sel_{p,r} \left\{t \mapsto \CT_{S(t)}(x(t),y(t))\right\}$ and let $D_{(u,v)}(t):=\{(u(t),v(t))\}.$ Then clearly the values of $D_{(u,v)}$ are uniform tangent sets a.e. and $g:=u, h:=v$ are trivial majorants of the values. Then $\Sel_{p,r} (D_{(u,v)})=\{(u,v)\}$ is uniform tangent set to $\Sel_{p,r}(S)$ at $(x,y).$ Therefore $(u,v) \in \CT_{\Sel_{p,r}(S)}(x,y).$
	\end{proof}
	
	We now move on to the global-to-pointwise inclusion.
	
	\begin{lemma} \label{lemma:sum-of-measurable-functions}
		Let $1 \leq p,r \leq \infty,$ let $R \in \Sigma$ satisfy $0 < \mu(R) < \infty$ and let measurable functions $f,g : R \to [0,\infty]$ satisfy
		$$
		f(t) + g(t) > c, \hspace{0.2cm} \text{a.e. $t \in R,$}
		$$
		where $c>0.$ Then
		$$
		\norm{f}_{L^p(R)} + \norm{g}_{L^r(R)} \geq \frac{c}{2} \min \left\{ \vartheta_p \left( \frac{\mu(R)}{2} \right), \vartheta_r \left( \frac{\mu(R)}{2} \right)\right\}.
		$$
	\end{lemma}
	\begin{proof}
		Let $R_f :=\{ t \in R : f(t) \geq \frac{c}{2}\}.$ If $\mu(R_f) \geq \frac{\mu(R)}{2}.$ Then $f \geq \frac{c}{2}$ on a set of measure at least $\frac{\mu(R)}{2}>0.$ In the case $p < \infty$ this yields
		\begin{align*}
			\norm{f}_{L^p(R)} \geq \left( \int_{R_f} f(t)^p \dd \mu(t) \right)^{1/p} \geq \frac{c}{2} \mu(R_f)^{1/p} \geq \frac{c}{2} \left( \frac{\mu(R)}{2} \right)^{1/p}.
		\end{align*}
		In the case $p=\infty$ simply $\norm{f}_{L^\infty(R)} \geq \frac{c}{2}.$ Combined $\norm{f}_{L^p(R)} \geq \frac{c}{2} \vartheta_p \left( \frac{\mu(R)}{2} \right).$
		
		Else $\mu(R_f) < \frac{\mu(R)}{2}.$ But then, by virtue of $R=R_f \cup (R \setminus R_f),$ we get
		\begin{align*}
			\mu(R) = \mu(R_f) + \mu( R \setminus R_f) <  \frac{\mu(R)}{2} + \mu(R \setminus R_f).
		\end{align*}
		Consequently $\mu(R \setminus R_f) > \frac{\mu(R)}{2}.$ But on $R \setminus R_f$ it holds $ f < \frac{c}{2}.$ Therefore 
		$$g > c - f > \frac{c}{2}$$ 
		on a set of measure atleast $\frac{\mu(R)}{2}>0.$ Repeating the argument from the first part of the proof yields $\norm{g}_{L^r(R)} \geq \frac{c}{2} \vartheta_r \left( \frac{\mu(R)}{2} \right).$ Combined with the nonnegativity of the norms the estimate in the conclusion follows.
	\end{proof}
	
	\begin{maintheorem} \label{theorem:productMainTheorem-global-to-pointwise-p-r}
		Let $1 \leq p,r \leq \infty, $ let $(T,\Sigma,\mu)$ be a complete $\sigma$-finite measure space, let $Y,Z$ be separable Banach spaces and  $S : T \rightrightarrows Y \times Z$ be a measurable with nonempty closed values. If $(x,y) \in \Sel_{p,r}(S),$ then
		$$
		\CT_{\Sel_{p,r}(S)}(x,y) \subset \Sel_{p,r} \left\{ t \mapsto \CT_{S(t)}(x(t),y(t))\right\}.
		$$
	\end{maintheorem}
	\begin{proof}
		Let $(u,v) \in \CT_{\Sel_{p,r}(S)}(x,y)$ and assume that there is a measurable $E \subset T$ with $\mu(E)>0$ such that $(u(t),v(t)) \notin \CT_{S(t)}(x(t),y(t))$ for each $t \in E.$ Using $\sigma$-finiteness, after passing to a suitable subset, we can assume $0 < \mu(E) < \infty.$ Fix a Castaing representation $\{s_j\}_{j=1}^\infty$ of $S$, and after modification on a null set assume
		$$
		s_j(t) \in S(t), \hspace{0.2cm} S(t) = \cl{ \{s_j(t) : j \in \NN \} }, \hspace{0.2cm} \forall t \in T.
		$$
		
		For each $m \in \NN$ define 
		\begin{equation*}
			B_{m}:=\left\{t \in E : \begin{array}{l} \forall N \in \NN \hspace{0.1cm} \exists L \in \NN \hspace{0.1cm} \forall \ell \geq L \hspace{0.1cm} \exists j \in \NN  \\ \norm{s_j(t)-(x(t),y(t))} < \frac{1}{2N} \\ d_{S(t)}\left( s_j(t) + \frac{1}{\ell} (u(t),v(t)) \right) > \frac{2}{m \ell} \end{array} \right\}.
		\end{equation*}
		The sets $B_m$ are measurable since
		\begin{align*}
		B_{m} &= E \cap \bigcap_{N=1}^\infty \bigcup_{L=1}^{\infty} \bigcap_{\ell = L}^{\infty} \bigcup_{j=1}^{\infty} \left\{t \in T  : \begin{array}{l} \norm{s_j(t) - (x(t),y(t)} < \frac{1}{2N} \\ \hspace{0.1cm} d_{S(t)}\left(s_j(t) + \frac{1}{\ell} (u(t),v(t)) \right) > \frac{2}{m \ell} \end{array}  \right\} \\
		&= E \cap \bigcap_{N=1}^\infty \liminf_{\ell \to \infty} \bigcup_{j=1}^{\infty} \left\{t \in T  : \begin{array}{l} \norm{s_j(t) - (x(t),y(t)} < \frac{1}{2N} \\ \hspace{0.1cm} d_{S(t)}\left(s_j(t) + \frac{1}{\ell} (u(t),v(t)) \right) > \frac{2}{m \ell} \end{array}  \right\}.
		\end{align*}
		and $t \mapsto \norm{s_j(t) - (x(t),y(t))},$ and $t \mapsto d_{S(t)}\left( s_j(t) + \frac{1}{\ell} (u(t),v(t)) \right)$ are measurable functions.
		Notice that
		\begin{equation} \label{eq:5-19}
			E = \bigcup_{m=1}^\infty B_m.
		\end{equation}
		Let $t \in E.$ Then $(u(t),v(t))$ is not a Clarke tangent at $(x(t),y(t)).$ By the Treiman's lemma \cite{TreimanLemma}, there exists $\varepsilon_t>0$ such that for every $\delta>0$ there exists $z \in S(t)$ and $\lambda>0$ satisfying
		\begin{equation} \label{eq:5-20}
			\norm{z - (x(t),y(t))} < \delta, \hspace{0.2cm} d_{S(t)}\left( z + \eta  (u(t),v(t)) \right) \geq \varepsilon_t \eta, \hspace{0.2cm} \forall \eta \in [0,\lambda].
		\end{equation}
		Choose $m \in \NN$ so large that $\frac{4}{m} < \varepsilon_t.$ Fix arbitrary $N \in \NN$ and apply (\ref{eq:5-20}) with $\delta = \frac{1}{4N}.$ Let $z$ and $\lambda$ be the corresponding points obtained in that way. Choose $L \in \NN$ so large that $\frac{1}{L} < \lambda.$ For every $\ell \geq L,$ the density of the Castaing representation allows us to choose an index  $j=j(\ell)$ such that
		$$
		\norm{s_j(t) - z} < \min \left\{ \frac{1}{4N}, \frac{1}{m \ell}\right\}.
		$$
		Then $\norm{s_j(t) - (x(t),y(t))} < \frac{1}{2N},$ and using $1$-Lipschitzness of the distance-function
		\begin{align*}
			d_{S(t)}\left( s_j(t) + \frac{1}{\ell}(u(t),v(t)) \right) &\geq d_{S(t)} \left( z + \frac{1}{\ell} (u(t),v(t)) \right) - \norm{s_j(t) - z} \\
			&>\frac{\varepsilon_t}{\ell} - \frac{1}{m \ell} \\
			&> \frac{4}{m \ell} - \frac{1}{m \ell}\\
			&> \frac{2}{m \ell}.
		\end{align*}
		Thus for the fixed $N,$ there exists $j=j(\ell)$ satisfying the two conditions in the definition of $B_m.$ Since $N$ was arbitrary we conclude $t \in B_m$ and so (\ref{eq:5-19}) holds.
		
		There exists $m_0 \in \NN$ such that $0< \mu(B_{m_0}) < \infty.$
		Set
		\begin{align*}
			M_{B_{m_0}}:&=\min\left\{ \vartheta_p \left( \frac{\mu(B_{m_0})}{4}  \right) , \vartheta_r \left( \frac{\mu(B_{m_0})}{4} \right) \right\} >0 \\
			\rho:&= \frac{M_{B_{m_0}}}{8 m_0}>0
		\end{align*}
		Since $(u,v)$ is a Clarke tangent to $\Sel_{p,r}(S)$ at $(x,y),$ for this $\rho>0$ there exists $\delta,\Lambda>0$ such that
		\begin{equation} \label{eq:5-27}
			\Sel_{p,r}(S) \cap [ z + \eta \left( (u,v) + \rho \clBB_{L^p \times L^r} \right) ] \neq \emptyset
		\end{equation}
		for every $z \in \Sel_{p,r}(S) \cap \left( (x,y) + \delta \clBB_{L^p \times L^r} \right)$ and every $\eta \in (0,\Lambda].$
		Choose $N_0 \in \NN$ so large that
		$$
		\frac{1}{N_0} \left( \vartheta_p(\mu(E)) + \vartheta_r (\mu(E))\right) < \delta.
		$$
		For each $\ell \in \NN$ define
		\begin{equation} \label{eq:5-29}
			R_{\ell}:= \left\{ t \in B_{m_0} :  \forall k \geq \ell \hspace{0.1cm} \exists j \in \NN, \begin{array}{l} \norm{s_j(t) - (x(t),y(t))} < \frac{1}{2N_0} \\
			d_{S(t)} \left( s_j(t) + \frac{1}{k} (u(t), v(t)) \right) > \frac{2}{m_0 k}  \end{array} \right\}.
		\end{equation}
		The sets $R_{\ell}$ are measurable, by the same argument as for the sets $B_m,$ and are increasing $R_{\ell} \subset R_{\ell+ 1}.$ Moreoever by the definition of $B_{m_0}$ with $N:=N_0,$ every $t \in B_{m_0}$ belongs to $R_{\ell}$ for all sufficiently large $\ell.$ Therefore
		$$
		R_{\ell} \nearrow B_{m_0}, \hspace{0.5cm} \mu(R_{\ell}) \nearrow \mu(B_{m_0})>0, \hspace{0.2cm} \ell \to \infty.
		$$
		Therefore we can choose $\ell_0 \in \NN$ so large that
		$$
		\frac{1}{\ell_{0}}< \Lambda, \hspace{0.5cm} \mu(R_{\ell_0}) > \frac{\mu(B_{m_0})}{2}.
		$$
		For $j \in \NN,$ define
		\begin{equation*}
			F_j := \left\{ t \in R_{\ell_0} : \begin{array}{l} \norm{s_j(t) - (x(t),y(t))} \leq \frac{1}{N_0} \\ d_{S(t)}\left(s_j(t) + \frac{1}{\ell_0}(u(t),v(t)) \right) > \frac{1}{m_0 \ell_0} \end{array} \right\}.
		\end{equation*}
		Clearly each $F_j$ is measurable and 
		$$R_{\ell_0} \subset \bigcup_{j=1}^\infty F_j.$$
		Now disjointify the family by setting
		$$
		\widehat{F}_1 :=F_1, \hspace{0.2cm} \widehat{F}_j :=F_j \setminus \bigcup_{i < j} F_i, \hspace{0.2cm} j \geq 2,
		$$
		and define the patched selector
		$$
		(\widehat{x}(t),\widehat{y}(t)) = \ind_{T \setminus R_{\ell_0}}(t) . (x(t),y(t)) + \sum_{j=1}^\infty \ind_{\widehat{F}_j}(t).s_j(t), \hspace{0.2cm} t \in T.
		$$
		Then $(\widehat{x},\widehat{y}) \in \Sel_{p,r}(S).$ Moreover, for $t \in R_{\ell_0}$
		$$
		\norm{(\widehat{x}(t),\widehat{y}(t)) - (x(t),y(t))} < \frac{1}{N_0}
		$$
		and therefore
		\begin{align*}
			\norm{ (\widehat{x},\widehat{y}) - (x,y)}_{p,r} &\leq \frac{1}{N_0} \left( \vartheta_p(\mu(R_{\ell_0}) + \vartheta_r(\mu(R_{\ell_0}) )  \right) \\
			&\leq \frac{1}{N_0} \left( \vartheta_p (\mu(E)) + \vartheta_r(\mu(E)) \right) < \delta.
		\end{align*}
		Thus $(\widehat{x},\widehat{y})$ is an admissible nearby base point in the sense of (\ref{eq:5-27}). 
		
		Let $(a,b) \in \Sel_{p,r}(S)$ be arbitrary. If $t \in R_{\ell_0},$ then $t \in \widehat{F}_{j} \subset F_j$ for a unique $j,$ and $s_j(t)= (\widehat{x}(t),\widehat{y}(t)).$ For any $c$ with $\norm{c}_{Y \times Z} \leq 1/2m_0$, the $1$-Lipschitzness of the distance function and the definition of $F_j$ give
		\begin{align*}
			d_{S(t)}\left( s_j(t) + \frac{1}{\ell_0}( (u(t),v(t)) + c ) \right) &\geq d_{S(t)}\left( s_j(t) + \frac{1}{\ell_0} (u(t),v(t)) \right) - \frac{\norm{c}_{Y \times Z}}{\ell_0} \\ 
			&> \frac{1}{m_0 \ell_0} - \frac{1}{2m_0 \ell_0} = \frac{1}{2m_0 \ell_0} >0.
		\end{align*}
		Hence
		$$
		S(t) \cap \left[s_j(t) + \frac{1}{\ell_0} \left( (u(t),v(t)) + \frac{1}{2m_0} \clBB_{Y \times Z} \right) \right] = \emptyset.
		$$
		Since $(a(t),b(t)) \in S(t)$ and $s_j(t)=(\widehat{x}(t),\widehat{y}(t))$ it follows that
		$$
		(a(t),b(t)) \notin (\widehat{x}(t),\widehat{y}(t)) + \frac{1}{\ell_0} \left( (u(t),v(t)) + \frac{1}{2m_0} \clBB_{Y \times Z}\right).
		$$
		Consequently for every $t \in R_{\ell_0}$ it holds
		$$
		\norm{\frac{a(t) - \widehat{x}(t)}{1/\ell_0} - u(t)}_{Y} + \norm{\frac{b(t) - \widehat{y}(t)}{1/\ell_0} -v(t) }_{Z}  > \frac{1}{2m_0}
		$$
		Now using Lemma \ref{lemma:sum-of-measurable-functions} we can estimate
		\begin{align*}
			\norm{\frac{a - \widehat{x}}{1/\ell_0} - u}_{L^p(R_{\ell_0})} &+ \norm{\frac{b - \widehat{y}}{1/\ell_0} -v }_{L^r(R_{\ell_0})}  \\
			&\geq \frac{1}{4m_0} \min\left\{ \vartheta_p \left( \frac{\mu(R_{\ell_0})}{2} \right), \vartheta_r \left( \frac{\mu(R_{\ell_0})}{2} \right) \right\} \\
			&\geq \frac{1}{4m_0} \min\left\{ \vartheta_p \left( \frac{\mu(B_{m_0})}{4} \right), \vartheta_r \left( \frac{\mu(B_{m_0})}{4} \right) \right\} \\
			 &= \frac{M_{B_{m_0}}}{4m_0} = 2 \rho.
		\end{align*}
		Since $(a,b)$ was arbitrary, the last estimate yields that for every $(a,b) \in \Sel_{p,r}(S)$ it holds
		$$
		\norm{ (a,b) - (\widehat{x},\widehat{y}) - \frac{1}{\ell_0} (u,v) }_{p,r} \geq \frac{2\rho}{ \ell_0}.
		$$
		Therefore
		$$
		\Sel_{p,r}(S) \cap [ (\widehat{x},\widehat{y}) + \frac{1}{\ell_0} \left( (u,v) + \rho \clBB_{L^p \times L^r} \right) ] = \emptyset.
		$$
		But $1/\ell_0 < \Lambda$ and so we have obtained a contradiction with (\ref{eq:5-27}) the fact that $(u,v)$ is Clarke tangent to $\Sel_{p,r}(S)$ at $(x,y)$ and $(\widehat{x},\widehat{y})$ is a constructed feasible nearby base point.
	\end{proof}
	
	 In particular when we take one of the variables to be zero we get  the following. 
	\begin{corollary} \label{th:Lp-to-pointwise}
		Let $(T,\Sigma,\mu)$ be a complete, $\sigma$-finite measure space, let $Z$ be a separable Banach space, $1 \leq p \leq \infty$ and $S : T \rightrightarrows Z$ be a measurable multivalued mapping with nonempty closed values. If $x \in \Hp(S),$ then
		\begin{enumerate}[(a)]
			\item $
			\left\{ v \in L^p ( T,Z) : v(t) \in \CT_{S(t)}(x(t)) \hspace{0.1cm} \text{a.e.}\right\} \supset \CT_{\Hp(S)}(x);
			$
			
			\item if aditionally $p <\infty,$ then $\left\{ v \in L^p (T,Z) : v(t) \in \CT_{S(t)}(x(t)) \hspace{0.1cm} \text{a.e.}\right\} = \CT_{\Hp(S)}(x).$
		\end{enumerate}
	\end{corollary}
	
	Let us remind that for $M>0, \hspace{0.1cm} (x,y) \in \Sel_{\infty,1}(S)$ we denote 
	$$\mathcal{U}_M:=\left\{ u \in M.\clBB_{L^\infty} : \exists v \in L^1(T,Z), \hspace{0.1cm} (u(t),v(t)) \in \CT_{S(t)}(x(t),y(t)) \hspace{0.1cm} \text{a.e.}  \right\}.$$
	Then $\Sel_{\infty,1}(S)$ satisfies the variational condition (\textbf{V}) at $(x,y)$ if for every $M>0, \hspace{0.1cm} \Sel_{\infty,1}(S)$ is $(\{0\} \times K)$-correctable along $(\mathcal{U}_M \times \{0\})$ at $(x,y)$ for some equi-integrable $K \subset L^1(T,Y).$
	\begin{proposition}
		Let $(T,\Sigma,\mu)$ be a complete, $\sigma$-finite measure space, let $Y,Z$ be separable Banach spaces, let $S : T \rightrightarrows Y \times Z$ be measurable with nonempty closed values and $(x,y) \in \Sel_{\infty,1}(S).$
		If $\Sel_{\infty,1}(S)$ satisfies the variational condition $(\textbf{V})$ at $(x,y),$ then
		$$
		\Sel_{\infty,1} \left\{ t \mapsto \CT_{S(t)}(x(t),y(t))\right\} = \CT_{\Sel_{\infty,1}(S)}(x,y). 
		$$
	\end{proposition}
	\begin{proof}
		Theorem \ref{theorem:productMainTheorem-global-to-pointwise-p-r} gives the global-to-pointwise inclusion unconditionally. It's suffices to prove that under (\textbf{V}), the pointwise-to-global inclusion holds at $(x,y).$ Take $(u,v) \in \Sel_{\infty,1} \left\{ t \mapsto \CT_{S(t)}(x(t),y(t)) \right\}$ and $M > \norm{u}_{L^\infty}.$ Then by the definition of $\mathcal{U}_{M}, \hspace{0.1cm} u \in \mathcal{U}_{M}.$ Condition (\textbf{V}) therefore gives $\delta,\lambda>0$ and a relatively weakly compact $K \subset L^1(T,Z)$ such that for every $\eta \in [0,\lambda]$
		$$
		\Sel_{\infty,1}(S) \cap \left( (x,y) + \delta M \clBB_{L^\infty \times L^1} \right) + \eta (u,0) \subset \Sel_{\infty,1}(S) - \eta (\{0\} \times K).
		$$
		Adding $\eta(0,v)$ to both sides we get that for every $\eta \in [0,\lambda]$
		$$
		\Sel_{\infty,1}(S) \cap \left( (x,y) + \delta M \clBB_{L^\infty \times L^1} \right) + \eta (u,v) \subset \Sel_{\infty,1}(S) - \eta (\{0\} \times (K-v)).
		$$
		Hence $\Sel_{\infty,1}(S)$ is $\{0\} \times (K-v)$-correctable along $\{(u,v)\}$ at $(x,y).$ But $K$ is relative weak compact in $L^1(T,Z),$ and hence so is $K-v.$ Consequently $\{0\} \times (K-v)$ is an equi-integrable subset of $L^\infty(T,Y) \times L^1(T,Z).$ Finall, put $D(t):=\{(u(t),v(t))\}.$ Then $D(t)$ is uniform tangent set to $S(t)$ at $(x(t),y(t))$ a.e. Then Theorem \ref{theorem:osnoven-rezultat-potochovo-kum-globalno} applies and shows that the set $\Sel_{\infty,1}(D)=\{(u,v)\}$ is uniform tangent to $\Sel_{\infty,1}(S)$ at $(x,y).$ That is $(u,v) \in \CT_{\Sel_{\infty,1}(S)}(x,y).$
	\end{proof}
	
	\begin{corollary}
		Let $(T,\Sigma,\mu)$ be a complete, $\sigma$-finite measure space, let $Y,Z$ be separable Banach spaces and $F : T \times Y \rightrightarrows Z$ be measurable in $t \in T$ and integrably Lipschitz in $y \in Y.$ That is there exist $\ell \in L^1(T)$ such that 
		$$F(t,y) \subset F(t,z) + \ell(t) \norm{y-z}\clBB_{Z}, \hspace{0.2cm} \text{for any $t \in T, y,z \in Y.$}$$
		If $(x,y) \in L^\infty(T,Y) \times L^1(T,Z)$ such that $y(t) \in F(t,x(t))$ a.e. Then
		$$
		\CT_{\Sel_{\infty,1} \left\{t \mapsto \Gr F(t,\cdot)\right\} }(x,y)=\Sel_{\infty,1} \left\{t \mapsto \CT_{\Gr F(t,\cdot)}(x(t),y(t))\right\}.
		$$
	\end{corollary}
	\begin{proof}
		Set $S(t):=\Gr F(t,\cdot)$ and $K_{\ell}:=\{k \in L^1(T,Z) : \norm{k(t)}_{Z} \leq \ell(t) \hspace{0.1cm} \text{a.e.}\}.$ Clearly $K_{\ell}$ is equi-integrable. Let $(x,y) \in \Sel_{\infty,1}(S), u \in \clBB_{L^\infty}, \eta >0$ be arbitrary. Using
		$$
		y(t) \in F(t,x(t)) \subset F(t, x(t) + \eta u(t)) + \eta \ell(t) \clBB_{Z} \hspace{0.1cm} \text{a.e.,}
		$$
		there exist $z_{\eta}(t) \in F(t, x(t) + \eta u(t))$ with $\norm{z_{\eta}(t) - y(t)}_{Z} \leq \eta \ell(t)$ a.e. Then $k_{\eta}:=(z_{\eta} - y)/ \eta \in K_{\ell}$ and $(x+ \eta u, y + \eta k_{\eta}) \in \Sel_{\infty,1}(S).$ Therefore 
		$$(x+\eta u, y) \in \Sel_{\infty,1}(S) - \eta (\{0\} \times K_{\ell}).$$ 
		Consequently we proved that for any $\delta,\lambda>0$ and $\eta \in [0,\lambda]$ it holds
		$$
		\Sel_{\infty,1}(S) \cap \left( (x,y) + \delta \clBB_{L^\infty \times L^1} \right) + \eta (\clBB_{L^\infty} \times \{0\}) \subset \Sel_{\infty,1}(S) - \eta (\{0\} \times K_{\ell}).
		$$
	\end{proof}
	
	Let us now start with some immediate applications.
	
	\section{Applications}
	
	If $1 \leq p < \infty,$ the conjugate exponent of $p$ is then $q$ such that $q=\infty$ whenever $p=1$ and $\frac{1}{p} + \frac{1}{q}=1$ else. If, additionally, $Z$ is reflexive $Z^{*}$ has the Radon-Nikodym property, hence on a $\sigma$-finite measure space and $1 \leq p < \infty,$ we can identify $[L^p(T,Z)]^{*}$ and $L^q(T,Z^{*})$ under the usual integral pairing
	$$\left<\xi,v\right>:=\int_T \left< \xi(t),v(t)\right> \dd \mu(t).$$
	
	Then we immediately obtain the following.
	\begin{corollary}
		Let $(T,\Sigma,\mu)$ be a complete $\sigma$-finite measure space, let $Z$ be a separable reflexive Banach space, $1 \leq p < \infty$ and $S: T \rightrightarrows Z$ be a measurable multivalued mapping with nonempty closed values and $x \in \Hp(S).$ Then
		\begin{enumerate}[(a)]
			\item $
			\CN_{\Hp(S)}(x)=\{\xi \in L^q(T,Z^{*}) : \xi(t) \in \CN_{S(t)}(x(t)) \hspace{0.1cm} \text{a.e.}\};
			$
			
			\item $\widehat{\partial}d_{\Hp(S)}(x) \subset \{ \xi \in L^q (T,Z^{*}) : \norm{\xi}_{L^q} \leq 1, \hspace{0.1cm} \xi(t) \in \CN_{S(t)}(x(t)) \hspace{0.1cm} \text{a.e.}\}.$
		\end{enumerate}
	\end{corollary}
	\begin{proof}
		The inclusion $\Hq(t \mapsto \CN_{S(t)}(x(t))) \subset \CN_{\Hp(S)}(x)$ follows immediately from the main result for the Clarke tangent cone. If $\xi \in L^q(T,Z^{*})$ with $\left<\xi(t),v(t)\right> \leq 0$ for every $v(t) \in \CT_{S(t)}(x(t))$ a.e. $t \in T.$ Then in light of Theorem \ref{thorem:Main-Lp-theorem} we get  $\left<\xi,v\right> \leq 0$ for every $v \in \CT_{\Hp(S)}(x).$
		
		Conversely. Pick $\xi \in \CN_{\Hp(S)}(x)$ and assume there exists measurable $A \subset T$ with $\mu(A)>0$ such that for every $t \in A$ there exists $w_t \in \CT_{S(t)}(x(t)) \cap \clBB_{Z}$ such that $$\left<\xi(t),w_t \right> > 0 .$$
		Thus $A=\bigcup_{n=1}^{\infty} A_n,$ where $A_n:=\left\{t \in T : \exists w_t \in \CT_{S(t)}(x(t)) \cap \clBB_{Z}, \hspace{0.1cm} \left<\xi(t),w_t \right> \geq \frac{1}{n}\right\}.$
		There exists possitive integer $n_0$ such that $\mu(A_{n_0})>0.$ Using $\sigma$-finiteness there exists measurable $E$ subset of $A_{n_0}$ with $0 < \mu(E) < \infty.$
		The multivalued mapping
		$$
		E \ni t \mapsto \left\{ w \in \CT_{S(t)}(x(t)) \cap \clBB_{Z} : \left<\xi(t),w \right> \geq \frac{1}{n_0} \right\}
		$$
		has nonempty measurable closed values on $E.$ Therefore it admits a measurable selector $w : E \to Z.$ Let $v:=\ind_{E}.w.$ Since $\mu(E)<\infty$ and $\norm{w(t)}\leq 1$ it holds $v \in L^p.$
		Further $v(t) \in \CT_{S(t)}(x(t))$ a.e. Thus from the proven theorem $v \in \CT_{\Hp(S)}(x).$ Consequently
		$$
		\left< \xi, v \right> = \int_E \left< \xi(t),w(t) \right> \dd\mu(t) \geq \frac{1}{n_0}.\mu(E) >0.
		$$
		Contrary to the fact that $\xi \in \CN_{\Hp(S)}(x).$
	\end{proof}
	
	\begin{corollary}[Lagrange multiplier rule]
		Let $(T,\Sigma,\mu)$ be a complete $\sigma$-finite measure space, let $Z$ be a separable reflexive Banach space, $1 \leq p < \infty$ and $S : T \rightrightarrows Z$ be a measurable multivalued mapping with nonempty closed values. Let $F : L^p(T,Z) \to \RR$ be locally Lipschitz and $\cl{x} \in \Hp(S)$ be a local solution of
		$$
		F(x) \to \min \hspace{0.5cm} \text{subject to} \hspace{0.5cm} x \in \Hp(S),
		$$
		then
		$$
		0 \in \widehat{\partial} F(\cl{x}) + \left\{\xi \in L^q(T,Z^{*}) : \xi(t) \in \CN_{S(t)}(\cl{x}(t)) \hspace{0.1cm} \text{a.e.} \right\}.
		$$
	\end{corollary}

	A natural source of decomposable sets in product spaces $L^p \times L^r$ is provided by graphs of Nemytskii operators. Let $Y,Z$ be separable reflexive Banach spaces, let $(T,\Sigma,\mu)$ be a complete $\sigma$-finite measure space and let $f: T \times Z \to Y$ be  such that $t \mapsto \Gr f(t,\cdot)$ is a measurable multifunction with nonempty closed values. This is the case whenever $f$ is Car\'atheodory. That is measurable in $t$ and continuous in $z.$ Assume, in addition, that the Nemytskii operator of $f$
	$$
	N_f : L^p(T,Z) \to L^r(T,Y), \hspace{0.3cm} N_f(x)(t):=f(t,x(t))
	$$
	is well defined, for $1 \leq p,r < \infty.$ This is the case under suitable growth assumptions on $f.$
	Then the graph of $N_f$ is naturally decomposable in the product
	$$\Gr N_f =\Hpr \left\{ t \mapsto \Gr f(t,\cdot) \right\}.$$ Thus $\CT_{\Gr N_{f}}(x,N_f(x))=\Hpr \left\{ t \mapsto \CT_{\Gr f(t,\cdot)}(x(t),f(t,x(t))) \right\}.$
	More generally, the multivalued analogue gives the following.
	\begin{corollary}[The Clarke cones to the graph of multivalued Nemytskii operator] \label{cor:nemytskii-application}
		Let $(T,\Sigma,\mu)$ be a complete $\sigma$-finite measure space, let $Z,Y$ be separable reflexive Banach space, $1 \leq p,r < \infty$, let $F : T \times Z \rightrightarrows Y$ be a multifunction such that $t \mapsto \Gr F(t,\cdot)$ is measurable with nonempty closed values and $N_F : L^p(T,Z) \rightrightarrows L^r(T,Y)$ be its multivalued Nemytskii operator
		$$
		N_F(x):=\{v \in L^r(T,Y): v(t) \in F(t,x(t)) \hspace{0.1cm} \text{a.e.}\}.
		$$
		Therefore $\Gr N_F = \Hpr( t \mapsto \Gr F(t,\cdot))$ and if $(x,y)\in \Gr N_F$ it holds
		\begin{enumerate}[(a)]
			\item $\CT_{\Gr N_{F}}(x,y) = \Hpr\left( t \mapsto \CT_{\Gr F(t,\cdot)}(x(t), y(t)) \right);$
			
			\item $\CN_{\Gr N_{F}}(x,y) = \Hqs\left( t \mapsto \CN_{\Gr F(t,\cdot)}(x(t), y(t)) \right)$ where $q,s$ are the conjugate exponents of $p,r$ respectively.
		\end{enumerate}
				\end{corollary}
				
		As we have seen, in case atleast one of the exponents is infinite, similar positive result holds under equi-integrable correction assumption.

	\section{Multiplier rule for  nonconvex integral programs under pointwise constraints}
	%Lagrange multiplier rules and nonconvex integral functionals

	In this section we record some consequences of the pointwise tangent-cone formula for nonsmooth integral problems. Throughout the section let
	$$
	\varphi : T \times Z \to \RR \cup \{+\infty\}
	$$
	be a normal integrand. That is, $\varphi$ is measurable in $t \in T$ and proper lower semicontinuous in $u \in Z.$
	Denote
	$$
	\Hpone \{t \mapsto epi \varphi(t,\cdot) \} :=\{(w,a) \in L^p(T,Z) \times L^1(T) : (w(t),a(t)) \in epi \varphi(t,\cdot) \hspace{0.1cm} \text{a.e.}\}.
	$$
	
	Denote the upper-integral functional (see \cite{Rockafellar1976Integral} and \cite{RockafellarWets1998})
	$$
	I_{\varphi}[x]:=\int_{T}^{*} \varphi(t,x(t)) \dd \mu(t):=\inf \left\{ \int_T \sigma \dd \mu : \sigma \in L^1(T), \hspace{0.1cm} \sigma(t) \geq \varphi(t,x(t)) \hspace{0.1cm} \text{a.e.} \right\}
	$$
	and consider the problem of minimizing an integral functional with pointwise constraints
	
	(\textbf{IP}) \hspace{2cm} $I_{\varphi}[x] \to \min \hspace{0.5cm} \text{subject to \hspace{0.2cm} $x \in \Hp(S).$}$
	
	\begin{definition}
		Let $A$ be a closed subset of the Banach space $X$ and $x_0 \in A.$ It's said that $A$ is compactly epi-Lipschitz at $x_0,$ if there exists $\delta,\lambda>0$ and a compact set $K \subset X,$ such that
		$$
		A \cap (x_0 + \delta \clBB_{X}) + \eta \varepsilon \clBB_{X} \subset A - \eta K, \hspace{0.5cm} \forall \eta \in [0,\lambda].
		$$
	\end{definition}

	Compact epi-Lipschitzness for a closed set is automatically satisfied in a finite-dimensional space since we can correct with the same small closed ball, which perturbs the base point.
	
	We construct infinite-dimensional multipliers for ($\textbf{IP}$), by integrating measurable selections of pointwise multipliers of the integrand.
	
	\begin{lemma}[See \cite{Giner1995Local} and \cite{Giner1995Contrainte} ] 
		Let $(T,\Sigma,\mu)$ be a complete, $\sigma$-finite, nonatomic measure space. Let $1 \leq p < \infty$ and $Z$ be a separable Banach space and $S : T \rightrightarrows Z$ be measurable multivalued mapping with nonempty closed values.
		
		If $x_0 \in \Hp(S)$ is local solution for $\textbf{(IP)}$ with $I_{\varphi}[x_0] \in \RR,$ then $$\varphi(t,x_0(t)) = \min_{u \in S(t)} \varphi(t,u), \hspace{0.2cm} \text{ a.e. \hspace{0.1cm} $t \in T$}.$$
	\end{lemma}
	
	\begin{maintheorem}
		Let $(T,\Sigma,\mu)$ be a complete, $\sigma$-finite, nonatomic measure space. Let $Z$ be a separable reflexive Banach space, let $1 \leq p < \infty,$ and let $q$ be the conjugate exponent of $p.$ Let $S :T \rightrightarrows Z$ be measurable with nonempty closed values, and let $\varphi: T \times Z \to \RR \cup\{+\infty\}$ be a normal integrand. Assume that
		\begin{enumerate}[(i)]
			\item $x_0$ is a local solution of $(\textbf{IP})$ with $I_\varphi[x_0] \in \RR;$ 
			
			\item $epi \varphi(t,\cdot)$ is compactly epi-Lipschitz at $(x_0(t),\varphi(t,x_0(t)))$ for a.e. $t \in T;$
			
			\item An arbitrary weight $\theta \in L^q(T) \cap L^\infty(T)$ with $\theta(t) >0$ a.e. is chosen. 
		\end{enumerate}
		Then there exists
		$$
		\xi \in L^q(T,Z^{*})  \hspace{0.5cm} \text{and} \hspace{0.5cm} \eta \in L^\infty(T)
		$$
		such that
		\begin{enumerate}[(a)]
			\item $0 \leq \eta(t) \leq \theta(t)$ and
			$\norm{\xi(t)} + \eta(t) = \theta(t)$ a.e. $t \in T;$
			
			\item $\int_T \left< \xi(t), v(t) \right> \dd \mu(t) \leq 0$ for every $v \in \CT_{\Hp(S)}(x_0);$
			
			\item $\int_T \left< \xi(t), w(t) \right> \dd \mu(t) + \int_T \eta(t) \sigma(t) \dd \mu(t) \geq 0$ \newline
			for every $(w,\sigma) \in \CT_{\Hpone(t \mapsto epi \varphi(t,\cdot))}(x_0(\cdot), \varphi(\cdot, x_0(\cdot))).$
		\end{enumerate}
	\end{maintheorem}
	
	\begin{proof}
		Let $x_0$ be \emph{local} solution of ($\textbf{IP}$). Then $x_0(t)$ is for a.e. $t \in T$ a global solution of
		$$
		\hspace{2cm} \varphi(t,u) \to \min \hspace{0.5cm} \text{subject to $u \in S(t)$.}
		$$
		Now $u \mapsto \varphi(t,u)$ is proper lower semicontinuous and $epi \varphi(t,\cdot)$ is compactly epi-Lipscthiz at $(x_0(t),\varphi(t,x_0(t)))$ a.e. $t \in T.$ Applying Corollary 5.6 from \cite{BivasKrastanovRibarska2020-tangential-transversality} we obtain that for almost every $t \in T$ there exists $(\widehat{\xi}_t, \widehat{\eta}_t) \in Z^{*} \times \{0,1\}$ such that
		
		\begin{align} \label{eq:three-pointwise-lagrange-multiplier-conditions}
			(\widehat{\xi}_t, \widehat{\eta}_t) &\neq (0,0) \\
			\left<\widehat{\xi}_t, v_t \right> &\leq 0, \hspace{0.2cm} \forall v_t \in \CT_{S(t)}(x_0(t)) \\
			\left<\widehat{\xi}_t, w_t \right> + \widehat{\eta}_t s_t &\geq 0, \hspace{0.2cm} \forall (w_t,s_t) \in \CT_{epi \varphi(t,\cdot)}(x_0(t),\varphi(t,x_0(t))).
		\end{align}
		We now claim that there exists measurable functions $\widehat{\xi}: T \to Z^{*}$ and $\widehat{\eta} : T \to \{0,1\}$ satisfying the three properties from above pointwise.
		Consider the multivalued mapping
		$$
		t \mapsto M(t):=(M_0(t) \times  \{0\}) \cup (M_1(t) \times \{1\}),
		$$
		where
		\begin{align*}
			M_0(t):&=\Big\{\zeta \in Z^{*} : \norm{\zeta}=1, \begin{array}{l} \hspace{0.1cm} \left<\zeta,v \right> \leq 0, \hspace{0.1cm} \forall v \in \CT_{S(t)}(x_0(t)), \\ \hspace{0.1cm} \left<\zeta , w \right> \geq 0,  \forall (w,s) \in \CT_{epi \varphi(t,\cdot)}(x_0(t),\varphi(t,x_0(t))) \end{array} \Big\} \\
			M_1(t) :&=\Big\{ \zeta \in Z^{*} : \begin{array}{l} \left<\zeta, v\right> \leq 0, \hspace{0.1cm} \forall v \in \CT_{S(t)}(x_0(t)), \\
				\left<\zeta, w \right> + s \geq 0, \hspace{0.1cm} \forall (w,s) \in \CT_{epi \varphi(t,\cdot)}(x_0(t),\varphi(t,x_0(t)))  \end{array} \Big\}
		\end{align*}
		The multifunctions $t\mapsto \CT_{S(t)}(x_0(t))$ and $t \mapsto \CT_{epi \varphi (t,\cdot)}(x_0(t),\varphi(t,x_0(t))$ are measurable with nonempty closed values and the duality brackets are continuous. Therefore $M_{0}(\cdot),M_{1}(\cdot)$ are measurable multivalued mappings with closed values such that $M_0(t)\cup M_1(t)$ is nonempty for a.e. $t \in T.$ Therefore $M(\cdot)$ is measurable with nonempty closed values. There exist measurable selector for $M(\cdot)$
		$$
		(\widehat{\xi}(\cdot),\widehat{\eta}(\cdot)) : T \to Z^{*} \times \{0,1\}
		$$ 
		
		Let $\widehat{\xi} : T \to Z^{*}$ and $\widehat{\eta} : T \to \{0,1\}$ be the obtained  measurable pointwise multipliers. We then pick for a.e. $t \in T$
		$$
		(\tilde{\xi}(t),\tilde{\eta}(t)):=\lambda(t) (\widehat{\xi}(t),\widehat{\eta}(t)), \hspace{0.2cm} \text{where} \hspace{0.2cm} \lambda(t):=\frac{1}{\norm{\widehat{\xi}(t)} + \widehat{\eta}(t)} >0.
		$$
		Since the second and third inequalities in the multiplier conditions are positively homogeneous we get that $(\widetilde{\xi}(t),\widetilde{\eta}(t))$ satisfy them as well a.e. That is for a.e. $t \in T$ it holds
		\begin{align}
			\norm{\tilde{\xi}(t)}+\tilde{\eta}(t) &=1, \\
			\left< \tilde{\xi}(t), v_t \right> &\leq 0, \hspace{0.2cm} v_t \in \CT_{S(t)}(x_0(t)) \\
			\left<\tilde{\xi}(t), w_t \right> + \tilde\eta(t) s_t &\geq 0, \hspace{0.2cm} (w_t,s_t) \in \CT_{epi \varphi(t,\cdot)}(x_0(t), \varphi(t,x_0(t))).
		\end{align}
		Let us note that the first condition, from above, implies that $\tilde{\xi} \in L^\infty(T,Z^{*})$ and $\eta \in L^\infty(T).$ So in particular in the case of finite measure space, this normalization of pointwise measurable multipliers produces summable infinite-dimensional ones. In general we can once again make use of positive homogenity and contorol the multipliers with a suitable weight. 
		
		Let $\theta \in L^q(T) \cap L^\infty(T)$ with $\theta(t) > 0$ a.e. Such weight exists for any $\sigma$-finite space and can be explicitly constructed. Indeed. If $p=1$ one could simply pick $\theta(t)=1.$ Else the conjugate exponent $q$ of $p$ is finite and let $T=\bigcup\{T_n : n \in \NN\}$ be a pairwise disjoint partition with $\mu(T_n)<+\infty.$ Then the following choice will suffice
		$$\theta(t):=\sum_{n=1}^\infty \frac{2^{-n}}{(1+\mu(T_n))^{1/q}} \ind_{T_n}(t).$$
		
		Then again by possitive homogenity pass to the multipliers
		$$
		(\xi(t),\eta(t)):=\theta(t)(\tilde{\xi}(t), \tilde{\eta}(t)), \hspace{0.2cm} \text{a.e.}
		$$
		Finally we conclude
		\begin{enumerate}[(a)]
			\item  $\norm{\xi(t)}+\eta(t)=\theta(t)$ and $0 \leq \eta(t) \leq \theta(t)$ a.e. \newline Consequently $\xi \in L^q(T,Z^{*})$ and $\eta \in L^\infty(T);$
			
			\item $\int_T \left< \xi(t), v(t) \right> \dd \mu(t) \leq 0$ for every $v \in \CT_{\Hp(S)}(x_0)$
			
			\item $\int_T \left< \xi(t), w(t) \right> \dd \mu(t) + \int_T \eta(t)\sigma(t) \dd \mu(t) \geq 0$ for every $(w,\sigma) \in L^p(T,Z) \times L^1(T)$ such that $(w(t),\sigma(t)) \in \CT_{epi \varphi(t,\cdot)}(x_0(t),\varphi(t,x_0(t)))$ a.e.
		\end{enumerate}
	\end{proof}

\end{document}